\documentclass[11pt]{amsart}
\usepackage[pagebackref]{hyperref}

\usepackage{amsthm,amsmath,amsxtra,amscd,amssymb,xypic,cleveref,xparse,xspace,mathtools}
\usepackage[all]{xy}
\usepackage{tikz-cd}
\usepackage{subfiles}
\usepackage{bm}
\usepackage{caption,subcaption,hyperref,cleveref}
\usepackage{txfonts} \usepackage{enumerate}

\theoremstyle{plain}
\newtheorem{thm}{Theorem}\newtheorem{lm}[thm]{Lemma}
\newtheorem{prop}[thm]{Proposition}

\newtheorem*{mthm*}{Main Theorem}

\newtheorem{claim}{Claim}

\theoremstyle{definition}

\newtheorem{remx}[thm]{Remark}
\newenvironment{rem}
{\pushQED{\qed}\remx}
{\popQED\endremx}

\AtBeginDocument{\def\MR#1{}
 }

\def\={\;=\;}  \def\+{\,+\,}

\newcommand{\GL}{\operatorname{GL}}

\newcommand{\divisor}[1]{{\rm div }\left( #1 \right)}

\newcommand{\calH}{\mathcal H}

\newcommand{\calM}{{\mathcal M}}

\newcommand{\CC}{{\mathbb{C}}}

\newcommand{\PP}{{\mathbb{P}}}
\newcommand{\QQ}{{\mathbb{Q}}}
\newcommand{\RR}{{\mathbb{R}}}
\newcommand{\ZZ}{{\mathbb{Z}}}

\newcommand{\poles}{\underline{p}}
\newcommand{\zeroes}{\underline{z}}

\DeclareDocumentCommand{\kmoduli}{ O{\mu} O{g} O{n} }{\calM_{#2,#3}(#1)}
\DeclareDocumentCommand{\barkmoduli}{ O{\mu} O{g} O{n} }{\overline{\calM}_{#2,#3}(#1)}
\DeclareDocumentCommand{\obarkmoduli}{ O{\mu} O{g} O{n} }{\Omega\overline{\calM}_{#2,#3}(#1)}

\DeclareDocumentCommand{\LMS}{ O{\mu} O{g,n}} {{\Xi\overline{\calM}_{#2}(#1)}}
\DeclareDocumentCommand{\RBLMS}{ O{\mu} O{g,n}}{\Xi\widehat{\calM}_{#2}(#1)}

\newcommand{\lG}{{\Gamma}}

\newcommand{\hor}{{\mathrm {hor}}}
\newcommand{\ver}{{\mathrm {ver}}}

\DeclareDocumentCommand{\Ehor}{O{\lG}}{E^{\hor}(#1)}   \DeclareDocumentCommand{\Ever}{O{\lG}}{E^{\ver}(#1)}  \DeclareDocumentCommand{\Ehori}{O{i} O{\lG} }{E^{\hor}_{(#1)}(#2)}   

\DeclareDocumentCommand{\abEhor}{O{\lG}}{\widetilde{E^{\hor}}(#1)}

\newcommand{\oHo}{{\overline{\calH^\circ}}}
\newcommand{\pHo}{\partial\overline{\calH^\circ}}

\DeclareDocumentCommand{\relhomZ}{O{X} O{\zeroes} O{\poles}}{H_1(#1\setminus #3,#2;\ZZ)}
\DeclareDocumentCommand{\relhomC}{O{X} O{\zeroes} O{\poles}}{H_1(#1\setminus #3,#2;\CC)}

\newcommand\Dhirr[1][]{{D^{h,{\rm irr}}_{#1}}}

\usepackage{color}

\makeatletter
\let\@wraptoccontribs\wraptoccontribs
\makeatother

\begin{document}

\title{Ends of the strata of differentials}
\date{\today}
\author[B.~Dozier]{Benjamin Dozier}
\thanks{Research of the first author is supported in part by the National Science Foundation under the grant DMS-22-47244, and by the Simons Foundation}
\address{Mathematics Department, Cornell University, Ithaca, NY 14853, USA}
\email{benjamin.dozier@cornell.edu}

\author[S.~Grushevsky]{Samuel Grushevsky}
\thanks{Research of the second author is supported in part by the National Science Foundation under the grant DMS-21-01631.}
\address{Department of Mathematics and Simons Center for Geometry and Physics, Stony Brook University, Stony Brook, NY 11794-3651}
\email{sam@math.stonybrook.edu}

\contrib[with an appendix by]{Myeongjae Lee}
\address{Department of Mathematics, Stony Brook University, Stony Brook, NY 11794-3651}
\email{myeongjae.lee@stonybrook.edu}

\begin{abstract}
We enumerate the ends of each stratum of meromorphic 1-forms on Riemann surfaces with prescribed multiplicities of zeroes and poles. Our proof uses degeneration techniques based on the construction in~\cite{BCGGMmsds} of the moduli space of multi-scale differentials, together with recent classification of connected components of generalized strata in \cite{lewo}.  In particular, from these results we  quickly deduce the theorem for holomorphic 1-forms, originally proved by Boissy~\cite{BoissyEnds}. 
\end{abstract}

\maketitle

\section{Introduction}

Consider the moduli space of Riemann surfaces of genus $g$ equipped with a meromorphic 1-form.  The multiplicative group $\CC^*$ acts on this space via multiplication of the differential by a complex number.  The resulting quotient is stratified by the combinatorial data of the multiplicities of zeros and poles.  The~{\em stratum $\calH_g(m_1,\dots,m_n)$} is the moduli space of data $(X,z_1,\dots,z_n,\omega)$ where $(X,z_1,\dots,z_n)\in\calM_{g,n}$ is a smooth Riemann surface with~$n$ distinct labeled marked points, and $\omega$ is a projective class of non-zero meromorphic $1$-forms on~$X$ such that its divisor of zeroes and poles is $\divisor\omega=\sum m_i z_i$. Here $m_1,\dots,m_n\in\ZZ$, and $\sum m_i=2g-2$. (The reason that we work with the projectivized differentials is that without this there is an obvious, non-interesting source of non-compactness coming from scaling). 

Perhaps the first result on the topology of the strata is due to Kontsevich and Zorich~\cite{kozo1} who enumerated the connected components of the {\em holomorphic} strata of 1-forms --- that is, for the case when all $m_i>0$. Boissy \cite{BoissyMeromCC} classified the connected components of the strata in the general case.

Strata of holomorphic 1-forms and strata of holomorphic quadratic differentials are never compact, as can be easily seen using the $\GL(2,\RR)$-action.  Recently, in particular with the development of degeneration techniques for Teichm\"uller dynamics, and with the investigation of the enumerative parallels between the strata of differentials and the double ramification loci, the loci of meromorphic differentials have become more prominent, even when approaching questions for the holomorphic case. 

\smallskip
Recall that for a topological space~$X$, the number of ends of $X$ is the limit of the number of connected components of~$X\setminus K_n$, for an exhaustion of~$X$ by compact subsets~$K_n$. The main result, Theorem 1.1, of Boissy's paper~\cite{BoissyEnds} is that every connected component of every stratum of holomorphic quadratic differentials has only one end. Since any connected component of a stratum of holomorphic 1-forms gives, by squaring, a connected component of the stratum of holomorphic quadratic differentials, it follows that each connected component of the stratum of holomorphic 1-forms also has only one end.

In this paper we use a compactification of $\calH_g(\mu)$ and degeneration techniques to study the ends of strata of arbitrary meromorphic strata. For the case of strata of holomorphic differentials, from our techniques we quickly re-deduce Boissy's result.

\begin{mthm*}
\label{mainthm}
Consider a connected component $\calH^\circ$ of a stratum $\calH_g(m_1,\dots,m_n)$ of projectivized differentials.  If $\dim \calH^\circ \ge 2$, then $\calH^\circ$ has exactly one end.
\end{mthm*}

\begin{rem}
    The number of ends when $\dim \calH^\circ < 2$ was known previously (the three cases below cover all possibilities):
    \begin{itemize}
        \item  If $g=0$, $n\le 3$, the number of ends equals $0$, since $\calH^\circ$ is compact.
    
        \item If $(g,n)=(0,4)$, the number of ends equals 3.  In this case, the stratum is connected, and simply isomorphic to $\calM_{0,4}$, so it has one end for each boundary point of $\overline{\calM}_{0,4}$ (recall that in our strata the zeroes and poles are labeled, so even if the orders of the differential are the same at some of them, these are distinct points in the stratum).

        \item If $(g,n)=(1,2)$, so that $(m_1,m_2)=(m,-m)$, recall from \cite{BoissyMeromCC} that the connected components of $\calH_1(m,-m)$ are classified by the rotation number $r$ dividing $m$ (the more general notion of rotation number $r\mid \gcd (m_1,\dots,m_n)$ classifying the components of $\calH_1(m_1,\dots,m_n)$ is also recalled in the appendix). Let then $\calH^\circ$ be the component of $\calH_1(m,-m)$ consisting of differentials with rotation number~$r$. Then the number of ends of $\calH^\circ$ equals
        $$
         \frac12\sum_{d\mid \tfrac{m}{r}}\phi(d)\,\phi\left(\tfrac{m}{rd}\right),
        $$ 
        where $\phi(n)$ is the Euler function counting the number of divisors of~$n$.  
        
        To see this, note that $\calH^\circ$ is isomorphic to the modular curve $X_1(m/r)$ (by taking the $r^{\text{th}}$ power of the differential). The ends of $\calH^\circ$ thus correspond to cusps of this modular curve, the number of which is known classically, and we refer to \cite[Thm.~2.4]{taharchamber} for a detailed discussion from the point of view of flat surfaces.
    \end{itemize}  
\end{rem}

The genus $g=0$ case of the Main Theorem will be easy, since each such stratum is isomorphic to $\calM_{0,n}$, which is known to have one end.

For all other strata, our main tool is the \emph{multi-scale compactification}, that is the moduli stack~$\LMS$ of multi-scale differentials.  It is a compactification of $\calH_g(\mu)$ constructed in~\cite{BCGGMmsds}. It is a stack with a normal crossing boundary divisor $\partial\LMS$. The boundary of the moduli space of multi-scale differentials is stratified, and each boundary stratum is essentially parameterized by a generalized stratum of differentials, that is, by a product of subsets of strata consisting of differentials satisfying residue conditions, together with some prong-matching data (see \cite[Sec.~6]{euler} for more details).

To prove our result, it suffices to show that the boundary $\partial\LMS$ of the multi-scale compactification is connected.  We will focus on one irreducible component of $\partial\LMS$, the ``irreducible type horizontal" boundary divisor $\Dhirr\subset\partial\LMS$. We will show that in the holomorphic case, $\Dhirr$ intersects any other irreducible boundary divisor, and in the meromorphic case we will show that those (very special) irreducible components of the boundary~$\partial\LMS$ that may not intersect~$\Dhirr$ themselves, do intersect other boundary divisors that intersect~$\Dhirr$. Then once we establish that $\Dhirr$ itself is connected, the proof that each connected component of each stratum has one end is complete.

For a connected component of the holomorphic stratum $\calH_g(m_1,\dots,m_g)$ with all $m_i\ge 0$, the boundary $\Dhirr$ is a connected component of the (usual, not generalized) stratum $\calH_{g-1}(m_1,\dots,m_g,-1,-1)$, and the connected components of meromorphic strata were classified by Boissy \cite{BoissyMeromCC}. For the meromorphic case, $\Dhirr$ is  the generalized stratum $\calH_{g-1}(m_1,\dots,m_g \mid -1,-1)\subset\calH_{g-1}(m_1,\dots,m_g,-1,-1)$ where we recall that the $\mid$ notation indicates that the residues at the two last simple poles are required to sum to zero. The connected components of generalized strata were just now classified by Lee and Wong \cite{lewo}, in full generality, and in particular for this case of relevance for us.

\begin{rem}
Perhaps the first interesting topological invariant of a (generalized) stratum~$\calH$ is its number of connected components, i.e.~$\dim H^0(\calH,\QQ)$, and the number of ends is perhaps the next natural topological invariant. For the approach we take here, via the moduli of multi-scale differentials, it is natural to ask to determine the combinatorics of the boundary complex of $\LMS$ --- its connectedness is equivalent to the fact that there is only one end, but its higher homology would shed further light on the homology of the strata, see \cite{cgp}. This topic will be pursued elsewhere.
\end{rem}

\subsection*{Acknowledgements}
We thank Mathematisches Forschungsinstitut Oberwolfach for hosting us in-person in June 2021, when we started thinking about this problem.  We are extremely grateful to Myeongjae Lee for enlightening discussions, and for providing the appendix.  We thank Juliet Aygun for corrections on an earlier draft.  

\section{Ends of strata of holomorphic 1-forms}
For motivation, and to show the power of using the degeneration methods, in this section we use $\LMS$ to quickly rederive the main result of~\cite{BoissyEnds} for the strata of holomorphic 1-forms. In the course of the proof, we will highlight, and label \Cref{claim:A}, \Cref{claim:B}, and \Cref{claim:C} the key parts of the argument that would need a modification, or a different proof, for the general case of meromorphic differentials.

We start by recalling that if the boundary of a compactification is connected, there is one end.  Here and in the sequel, $\oHo$ denotes the closure in the multi-scale space $\LMS$ of a connected component $\calH^\circ$ of a stratum of projectivized differentials.

\begin{lm}
    \label{lm:bdy-cpx-end}
    For any $\calH^\circ$, if $\pHo$ is non-empty and connected, then $\calH^\circ$ has exactly one end.  
\end{lm}
\begin{proof}
Let $K\subset \calH^\circ$ be compact, and let $U \coloneqq \oHo-K$, which is open. By taking the union of open balls centered at each point of $\pHo$, we can construct an open set $U'\subset U$ with $U'\supset \pHo$. Since $\pHo$ is assumed to be connected, and $U'$ is a union of open balls centered at its point, $U'$ must also be connected. By \cite{BCGGMmsds}, the compactification $\oHo$ is a smooth, compact, complex orbifold, and so since $\pHo$ has positive (complex) codimension, $U'\setminus\pHo$ is also  connected. Now the complement $\calH^\circ-U'$ is a compact set $K'$ that contains $K$.  The complement of $K'$ in $\calH^\circ$ is exactly $U'-\pHo$. Thus any such $K$ lies in a larger compact $K'$ whose complement has a single component.  This implies that $K$ has one unbounded component, and we are done.
\end{proof}

\begin{thm}[\cite{BoissyEnds}, the case of abelian differentials]\label{thm:abelian}
Any connected component $\calH^\circ$ of any stratum $\calH_g(m_1,\dots,m_n)$ of {\em holomorphic} 1-forms has only one end.
\end{thm}
While the techniques Boissy uses to prove this result revolve around certain constructions with flat surfaces, we use the compactification directly. Before giving the proof, we recall the relevant properties of the boundary of $\LMS$ that were obtained in \cite{BCGGMmsds}, also referring to the information from~\cite{euler} as needed; this also allows us to introduce the relevant notation used throughout the paper.

Let $\oHo$ be the moduli space of multi-scale differentials compactifying $\calH^\circ$ (i.e.~the connected component of $\LMS$ that contains $\calH^\circ$). We will show that the boundary $\pHo\coloneqq\oHo\setminus\calH^\circ$ is connected, which by the above implies that the open stratum $\calH^\circ$ has only one end. 

Indeed, $\pHo$ is a normal crossing divisor, all of its irreducible components are divisors, and to prove that $\pHo$ is connected, it suffices to show that for any pair of irreducible components of~$\pHo$, there exists a path in~$\pHo$ connecting a generic point on one component to a generic point on the other component.

Recall that the codimension of a boundary stratum of $\oHo$ corresponding to some enhanced level graph $\Gamma$ is equal to the number of levels of $\Gamma$ below the top level, plus the number of horizontal edges in~$\Gamma$ (those whose endpoints are vertices at the same level). Thus the boundary divisors of $\oHo$ are indexed by enhanced two-level graphs~$\Gamma$ with no horizontal edges, with a given prong-matching equivalence class, and by dual graphs that only have one edge, which is horizontal. The irreducible components of boundary divisors of the first kind are called {\em vertical}, and we will enumerate them as $D_i$, where $i\in I$ runs over the set of all such irreducible components in $\pHo$ (so $I$ indexes some set of two-level enhanced graphs with no horizontal edges, and with a chosen equivalence class of prong-matchings). The boundary divisors corresponding to enhanced level graphs that only have one edge, which is a horizontal edge, are called {\em horizontal}.

Recall that the twisted differential must have simple poles at the two preimages of the node corresponding to the horizontal edge --- for degree reasons it cannot be holomorphic at these preimages of the node. In particular, if $\Gamma$ is a level graph that has only one edge, which is horizontal, then for the case of a holomorphic 1-form this edge must be non-separating, as otherwise the residues would have to be zero by the residue theorem. Thus any twisted differential on a generic point of the horizontal boundary of $\oHo$ is a meromorphic differential with two simple poles on an irreducible nodal curve with one node.  We denote $\Dhirr\subset \pHo$ (``h" for horizontal, and ``irr" because the corresponding nodal curve is irreducible) the corresponding boundary points, and call it {\em irreducible type horizontal} boundary. We are now ready to give a quick proof of the theorem.

\begin{proof}[Proof of \Cref{thm:abelian}]
First, recall that there is no prong-matching at a horizontal edge, and thus $\Dhirr$ is simply the quotient of a connected component of the subspace of the stratum $\calH_{g-1}(-1,-1,m_1,\dots,m_n)$ by the involution that interchanges the two poles. However, a priori it is not even clear that $\Dhirr$ is non-empty, and this is the first statement we need to prove.

\begin{claim}
\label{claim:A}
For any connected component $\calH^\circ$ of any stratum of holomorphic 1-forms, the boundary divisor $\Dhirr$ is non-empty.
\end{claim}

To prove this, consider any deepest boundary point $X$ of $\pHo$, i.e. one lying in a boundary stratum that cannot be further generated.  Let $\Gamma$ be the corresponding enhanced level graph.  Let $v$ be some top-level vertex of $\Gamma$.  The differential $X_v$ here lies in some generalized stratum of differentials.  

If $X_v$ has at least one horizontal edge attached to it, then we can smooth all level transitions in~$\Gamma$ and all other horizontal edges in~$\Gamma$ besides this horizontal edge, to obtain a boundary point in $\Dhirr$  (recall that the containments of boundary strata of $\LMS$ are determined by graph degenerations; see also \cite[Prop.~6.9 and Cor.~6.10]{BeDoGr} for much more general statements).

The other case is when $X_v$ has no horizontal nodes, and thus the differential on it must be holomorphic.  In this case, the generalized stratum that $X_v$ lies in must be $0$-dimensional, since otherwise it could be further degenerated (say, using the $\GL(2,\RR)$-action), which would be a contradiction.  But there are no holomorphic $0$-dimensional strata, so we are done.  

While the above is a direct proof of \Cref{claim:A}, there is of course another approach to this statement. Indeed, Boissy \cite{BoissyMeromCC} classified
connected components of strata of meromorphic differentials, and in particular showed that the connected components of the stratum $\calH_{g-1}(m_1,\dots,m_n,-1,-1)$ are in bijection with connected components of the holomorphic stratum $\calH_g(m_1,\dots,m_n)$. More precisely, he shows that plumbing together the two simple poles for a differential in~$\calH_{g-1}(m_1,\dots,m_n,-1,-1)$ gives a differential in~$\calH_g(m_1,\dots,m_n)$ --- and since for both strata the connected components are classified by the same spin and hyperellipticity invariants, this plumbing establishes a bijection on the sets of their connected components. This implies \Cref{claim:A} and also

\begin{claim} 
\label{claim:B}
The irreducible type horizontal boundary divisor $\Dhirr\subset\pHo$ in a holomorphic stratum is connected.
\end{claim}

To complete the proof of \Cref{thm:abelian}, we need the following statement:

\begin{claim}
\label{claim:C}
The irreducible type horizontal boundary divisor $\Dhirr$ has a non-empty intersection with any irreducible component~$D_i$ of the vertical boundary of~$\pHo$.    
\end{claim} 

Once this is established, this means that $\Dhirr\cup D_i$ is connected for each $i$, and thus $\pHo=\cup_{i\in I} (D_i\cup\Dhirr)$ is also connected, as a union of connected topological spaces, any pair of which intersects --- so that the theorem is proven.

To prove \Cref{claim:C}, we proceed similarly to the direct approach to \Cref{claim:A} above. Recall that $D_i$ parameterizes a connected component of some product of strata, with prongs chosen and residue conditions imposed, with one factor for each vertex of the corresponding two-level graph $\Gamma$, which has no horizontal edges. In particular, the factor corresponding to $X_v$ for some top level vertex~$v$ of~$\Gamma$ is a connected component $\calH_v$ of some stratum~$\calH_{g_v}(\mu_v)$. Since $\mu_v$ is a union of a subset of $\mu$ and possibly some prescribed multiplicities of zeroes at the vertical nodes, $\calH_{g_v}(\mu_v)$ is a connected component of some stratum of {\em holomorphic} 1-forms, with some choice of prongs (and with no residue conditions as it is at the top level).

But now by \Cref{claim:B}, we know that in the multi-scale compactification $\overline{\calH_v}$ of~$\calH_v$ there is a non-empty irreducible type horizontal boundary divisor, which we denote $\Dhirr[v]\subset\partial\overline{\calH_v}$. But what this means is that we can further degenerate a generic point of $D_i$ by inserting a horizontal loop at the vertex $v$, and thus the resulting enhanced level graph will also lie in $\Dhirr$, so that we have constructed a point in the intersection $D_i\cap\Dhirr$.
\end{proof}

\section{Ends of general meromorphic strata}
To prove the result for general meromorphic strata, we will have to prove or bypass the analogs of each of \Cref{claim:A}, \Cref{claim:B}, \Cref{claim:C} in the proof above. We first explain the ideas of dealing with each of these, before proceeding with the actual proof.

\begin{itemize}
    \item Analog of \Cref{claim:A}: This statement is simply false for meromorphic strata in genus zero (recall that there do no exist holomorphic strata in genus zero). Indeed, the dual graph of a stable curve of genus zero must be a tree, which in particular cannot contain any loops. However, in \Cref{lm:Dhirr_exists} below we will show that for any connected component of any meromorphic stratum in any {\em positive} genus the irreducible type horizontal boundary divisor $\Dhirr$ is still non-empty.

    \item Analog of \Cref{claim:B}: The boundary of $\pHo$ in general will consist of vertical boundary divisors $D_i$, the irreducible type horizontal boundary divisor $\Dhirr$, but also boundary divisors $D_j^h$, a generic point of one of which corresponds to a level graph with two vertices and one (separating) horizontal edge connecting them. In~\Cref{lm:Dh_inter_horiz} we will show that $\Dhirr\cap D_j^h\ne\emptyset$ for any~$j$, while the recent results of Lee and Wong \cite{lewo} in particular imply that~$\Dhirr$ is irreducible for any $g>1$, as we record in~\Cref{thm:Dh_isirr}. In the appendix, Myeongjae Lee shows that connected components of $\Dhirr$ in genus 1 can be connected to each other within $\pHo$.

    \item Analog of \Cref{claim:C}: If we try to show that $\Dhirr$ intersects every vertical boundary divisor $D_i\subset\pHo$ by applying the same argument as for the holomorphic case, by \Cref{claim:B} it will work unless every vertex $v\in V(\Gamma_i)$ of the corresponding vertical two-level graph $\Gamma_i$ corresponds to a curve of genus~$g_v=0$. The case of vertical two-level graphs with all vertices of genus zero will be a special case for us, and we deal with it by applying further level degenerations and undegenerations, to create a vertex with higher genus. In this process we will connect the corresponding vertical boundary divisor of $\pHo$ to a different vertical boundary divisor, which will then intersect $\Dhirr$, completing the proof. 

\end{itemize}

\smallskip
We now proceed to the proof; we first show that the irreducible type horizontal boundary divisor exists in general, proving the analog of \Cref{claim:A}. This, and the proof of analog of \Cref{claim:B} due to \cite{lewo}, are the two elements of our proof that use flat-geometric constructions, rather than navigating the boundary of the multi-scale compactification. 
\begin{lm}\label{lm:Dhirr_exists}
For any $g>0$, and for any connected component $\calH^\circ$ of any {\em possibly meromorphic} stratum $\calH_g(m_1,\dots,m_n)$, there exists an irreducible type horizontal boundary divisor: $\emptyset\ne\Dhirr\subset\pHo$.
\end{lm}
\begin{rem}
To further understand the topology of the strata via the homology of the boundary complex, it is natural to try to extend our work to the fully generalized strata, i.e.~allowing disconnected Riemann surfaces, and imposing suitable residue conditions. The proof below does not directly apply in this case, as the flat deformation constructed, of suitably stretching the handle that was bubbling off, can change the residues, and thus may not preserve the residue conditions imposed. We do not know if this Lemma holds for all generalized boundary strata.
\end{rem}

While below we give a direct proof of the lemma, by degeneration, we note that this result can also be recovered, with some extra work, from the results of \cite{lewo} on classification of connected components of generalized strata. Indeed, by \cite{lewo} one can enumerate the connected components of $\calH_{g-1}(m_1,\dots,m_n\mid -1,-1)$.  Keeping track of the geometric invariants determining the component, they are observed to be in bijection with connected components of~$\calH_g(m_1,\dots,m_n)$, similar to the holomorphic case.

\begin{proof}
For $\calH^\circ$ a component of a minimal stratum (i.e.~just a single zero, and any number of poles), Boissy ~\cite[Prop.~6.1]{BoissyMeromCC} gives that there is a surface in $\calH^\circ$ that is obtained by bubbling a handle from a genus~$g-1$ surface. This surface must then have a non-separating cylinder (the handle). Stretching this cylinder so that its modulus goes to infinity yields a limit point in $\Dhirr$. 

\begin{figure}[h]
    \centering
    \includegraphics[scale=0.7]{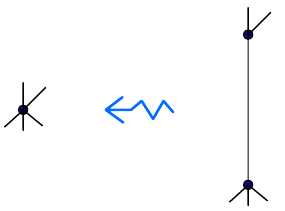}
    \caption{Level graphs corresponding to a vertical degeneration that merges all the zeroes together. The arrow $\leftsquigarrow$ denotes the undegeneration of differentials; it points in the direction of the corresponding morphism of enhanced level graphs, while the differentials specialize in the opposite direction.}
    \label{fig:MergeZeroes}
    \end{figure}

For a connected component $\calH^\circ$ of an arbitrary stratum, by~\cite[Prop.~7.1]{BoissyMeromCC}~$\calH^\circ$ is adjacent to some component $\calH'$ of a minimal stratum.  This means $\calH'$ can be thought of as lying in $\pHo$\footnote{Indeed, in terms of the multi-scale compactification, we are considering a boundary stratum where all the marked points have collided; the dual graph has two levels, with a genus~$g$ vertex with no marked points on the top level, and a genus zero vertex with all the marked points at the bottom level, as in \Cref{fig:MergeZeroes}}. By the above,~$\calH'$ admits an irreducible horizontal degeneration. Since~$\calH'$ is thought of as parameterizing the stratum at a top level vertex, we can consider the degeneration in $\pHo$ where the vertex corresponding to $\calH'$ has been degenerated to add an extra horizontal loop. Smoothing out the level passage from $0$ to $-1$ in $\LMS$ then constructs a point in $\Dhirr\subset\LMS$.
\end{proof}
Now that we know that $\Dhirr\ne\emptyset$, we ask whether it is connected. While in the holomorphic case the connectedness follows from Boissy's classification \cite{BoissyMeromCC} of connected components of meromorphic strata with two simple poles, and a number of zeroes, in the meromorphic case $\Dhirr$ is a generalized stratum of differentials, as the condition of opposite residues at the two new simple poles needs to be imposed. The recent work of Lee and Wong \cite{lewo} investigates the connected components of arbitrary generalized strata of differentials, expanding upon and generalizing the techniques of Lee \cite{Lee} for residueless strata. In particular, Theorem~1.4 and Theorem~1.6 in \cite{lewo} for the case of generalized strata where the only residue condition is that a pair of simple poles have opposite residues can be rephrased to give the analog of \Cref{claim:B} for the meromorphic case:
\begin{thm}[Lee-Wong]\label{thm:Dh_isirr}
For any $g>1$ and for any connected component $\calH^\circ$ of any meromorphic stratum, the irreducible type horizontal boundary divisor $\Dhirr\subset\overline\calH^\circ$ is irreducible. 
\end{thm}

As warm-up for the proof of the main theorem, we first show that $\Dhirr$ intersects all other horizontal boundary divisors, in the general meromorphic case.

\begin{lm}\label{lm:Dh_inter_horiz}
For any $g>0$, and any $D_j^h\subset\pHo$, an irreducible component of a horizontal boundary divisor, the intersection $D_j^h\cap\Dhirr$ is non-empty,
\end{lm}

\begin{proof}
By~\Cref{thm:Dh_isirr}, $\Dhirr$ is irreducible, and thus the nodal Riemann surface underlying a generic point of $D_j^h$ cannot be irreducible. Thus the generic point of~$D_j^h$ must have a dual graph as in \Cref{fig:HorizDegens}: it has two vertices, of genera $g_1$ and $g_2$, with marked zeroes and poles of multiplicities $\mu_1$ and $\mu_2$, and another simple pole at each of the preimages of the node, connected by one horizontal edge, such that $g=g_1+g_2$ and $\mu=\mu_1\sqcup \mu_2$.
\begin{figure}[h]
    \centering
    \includegraphics[scale=0.7]{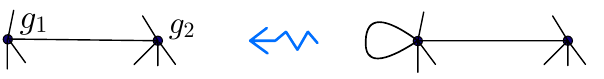}
    \caption{Level graphs corresponding to a degeneration of a point of $D_j^h$ to a point of $D_j^h\cap\Dhirr$.}
    \label{fig:HorizDegens}
\end{figure}
Denoting $\calH_i\coloneqq\calH_{g_i}(\mu_i,-1)$, the corresponding stratum of twisted differentials is then a connected component of the subset of the product $\calH_1\times\calH_2$ given by the condition that the two residues at the two preimages of the horizontal node are opposite (recall that there are no prongs at horizontal edges). Since $g=g_1+g_2>0$, without loss of generality assume $g_1>0$. But then by~\Cref{lm:Dhirr_exists} there exists an irreducible type horizontal boundary divisor $\emptyset\ne\Dhirr[1]\subset\overline{\calH}^\circ_1$. By definition of $\Dhirr\subset\oHo$, any point in $\Dhirr[1]$, together with any point in $\calH_2$, with the differential rescaled to have the opposite residue at the horizontal node, is then a point of $\Dhirr\cap D_j^h$.
\end{proof}

Recall that $D_i$ denotes an irreducible component of a vertical boundary divisor in  $\pHo$.
\begin{lm}
    \label{lm:horiz-vert}
    Suppose $\calH^\circ$ has dimension $d\ge 2$ and parameterizes surfaces of genus $g>0$. Then for a generic point in $D_i$, there is a path in $\pHo$ starting there and ending in ~$\Dhirr$. 
\end{lm}

\begin{proof}

In the holomorphic case, we argued that $\Dhirr$ will have a non-empty intersection with every $D_i$. Here this may not be true for some irreducible components $D_i$ of the vertical boundary, and in those cases we will have to consider a sequence of degenerations and undegenerations, which will ultimately allow us to navigate from~$D_i$ to~$\Dhirr$ along the boundary~$\pHo$. 

Recall that~$D_i$ parameterizes (a connected component of) the space of twisted differentials compatible with a certain enhanced level graph~$\Gamma$ (with a certain prong-matching equivalence class, which we fix once and for all), which has two levels $0$ and $-1$, and all of whose edges are vertical, connecting a vertex at level $0$ with a vertex of level $-1$. Recall that $g$ is equal to the sum of the genera of the vertices of $\Gamma$ plus the first Betti number $h_1(\Gamma)$.

\medskip

\underline{Case 1.} The easy case is if $\Gamma$ has at least one vertex $v\in V(\Gamma)$ such that the genus of the corresponding component $X_v$ of the stable curve is $g_v>0$.  \Cref{lm:pos-genus} resolves this  case.  

\medskip

\underline{Case 2.} The difficulty in the meromorphic case arises for irreducible boundary divisors~$D_i\subset\LMS$ where {\em every} vertex $v\in V(\Gamma)$ has genus $g_v=0$; see \Cref{fig:VertexUp} for an example. By \Cref{lm:no-resid-cond}, there are no global residue conditions, and thus each vertex corresponds to the (usual, not generalized) genus zero stratum $\calH_0(\mu_v)$.  This, however, does not admit a degeneration where the stable curve is irreducible (i.e.~each $\Dhirr[v]=\emptyset$).

    \begin{figure}[h]
    \centering
    \includegraphics[scale=0.7]{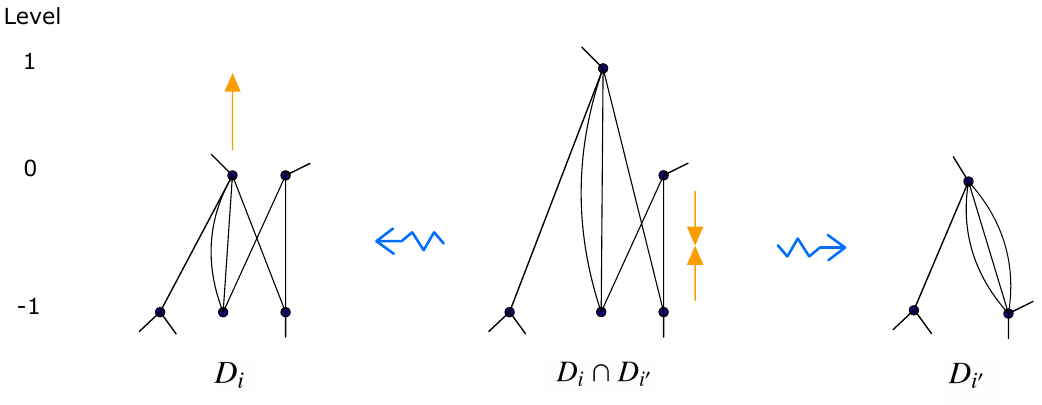}
    \caption{Degeneration and undegeneration used in Case 2 to reduce to the case of a single top level vertex.}
    \label{fig:VertexUp}
    \end{figure}
    
If the top level contains more than one vertex, we pick one top level vertex, and consider the degeneration of the level graph, which introduces level 1 above level 0, and pulls this vertex there --- without changing the dual graph, partial order on the vertices, or enhancements. We thus obtain a graph with levels 1,0,-1, and then consider its undegeneration that collapses the level transition between levels 0 and -1. The top level 
of the resulting graph has a single vertex, while the graph has two levels, and thus parameterizes some irreducible component $D_{i'}$ of the boundary $\pHo$. If any vertex in this new graph has genus $g_v>0$, then by \Cref{lm:pos-genus}, we can connect $D_{i'}$ to $\Dhirr$, and since by construction of the common degeneration we also have $D_i\cap D_{i'}\ne\emptyset$, we see that any point of~$D_i$ can be connected to $\Dhirr$.

If all vertices of $D_{i'}$ still have genus zero, but its bottom level has more than one vertex, we argue as in \Cref{fig:VertexDown}.  First we pull one bottom level vertex to level -2, and then in the resulting three-level graph collapse the level transition between levels~$0$ and $-1$, so that we construct a point in~$D_{i'}\cap D_{i''}$ for some boundary divisor $D_{i''}$ where the corresponding two-level graph~$\Gamma''$ has a unique vertex at each level. Again, if the genus of either vertex of~$\Gamma''$ is positive, we can conclude using \Cref{lm:pos-genus}.  

So we have reduced to the case of a vertical level graph with two vertices, at different levels, both of genus $0$.  Then \Cref{lm:two-vert} gives the required path.  

    \begin{figure}
    \centering
    \includegraphics[scale=0.7]{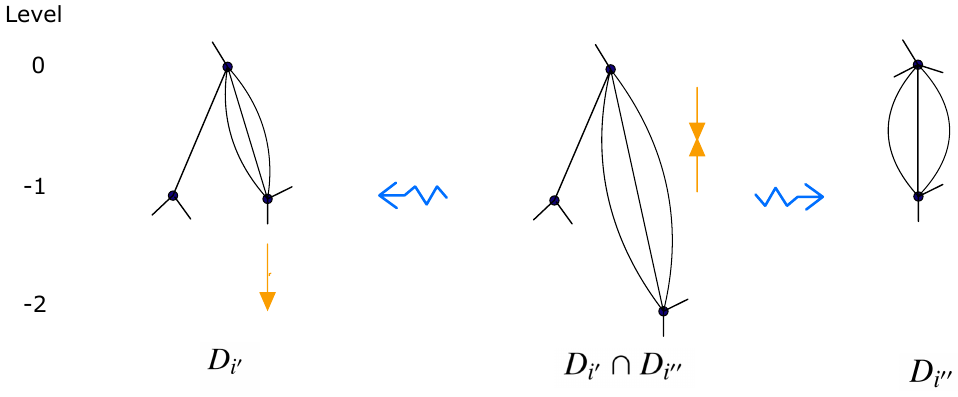}
    \caption{Degeneration and undegeneration used in Case 2 to reduce to the case of a single bottom level vertex.}
    \label{fig:VertexDown}
    \end{figure}

\end{proof}

\begin{lm}
    \label{lm:no-resid-cond}
    Let $D_i$ correspond to a level graph $\Gamma$ with all top level vertices of genus $0$.  Then there are no global residue conditions on the twisted differentials in $D_i$.  
\end{lm}

\begin{proof}
    Note that any top level genus zero vertex must contain a marked pole --- as all edges of~$\Gamma$ go down from that vertex and correspond to points where the order of the twisted differential is non-negative, while the total number of zeroes and poles of a meromorphic differential on $\PP^1$ is equal to $-2$.  Thus such a vertex does not impose a global residue condition, and since we are assuming all top level vertices have genus zero (and there are only two levels, since $D_i$ is a divisor), there are no global residue conditions at all.    
\end{proof}

\begin{lm}
    \label{lm:pos-genus}
    Suppose $\calH^\circ$ parameterizes surfaces of genus $g>0$, and let $D_i$ correspond to a level graph $\Gamma$ that has at least one vertex of positive genus. Then for a generic point of $D_i$, there is a path in $\pHo$ starting at this point and ending in $D^{h,irr}$.  
\end{lm}

\begin{proof}
    First we consider the case in which some top level vertex $v\in V(\Gamma)$ (i.e.~such that $\ell(v)=0$) has $g_v>0$.  Since it's at the top, there are no global residue conditions on this component, and the corresponding twisted differentials on $X_v$ are parameterized by a connected component of some (usual, not generalized) stratum $\calH_{g_v}(\mu_v)$. By \Cref{lm:Dhirr_exists}, there exists then the irreducible type boundary divisor $\emptyset\ne\Dhirr[v]\subset\overline{\calH}_{g_v}(\mu_v)$. Taking a point in this divisor, and an arbitrary twisted differential on the rest of the stable curve in $D_i$ then gives a point in $D_i\cap\Dhirr$.  

    It remains to consider the case in which no top level vertex has positive genus.  By \Cref{lm:no-resid-cond}, then there are no global residue conditions at all.  Let $v$ be a bottom level vertex $v$ (i.e.~$\ell(v)=-1$) with $g_v>0$.  Since there are no global residue conditions, again the twisted differentials on $X_v$ are parameterized by a connected component of some stratum $\calH_{g_v}(\mu_v)$, with no residue conditions imposed. This stratum again has a non-empty horizontal irreducible boundary divisor~$\Dhirr[v]$ by \Cref{lm:Dhirr_exists}, and the same argument as above applies.
\end{proof}

\begin{lm}
\label{lm:two-vert}
    Suppose $g\ge 1$ and the dimension of the stratum is at least~$2$. Let $D_i$ correspond to a level graph $\Gamma$ with two vertices (at different levels), both of genus zero. Then for a generic point of $D_i$, there is a path in $\pHo$ starting at this point and ending in $D^{h,irr}$.  
\end{lm}

\begin{proof}
    First note that by \Cref{lm:no-resid-cond}, there are no global residue conditions for $D_i$.  Thus each vertex of $\Gamma$ parameterizes genus zero meromorphic differentials with no residue conditions, and hence the stratum of projectivized differentials corresponding to each vertex is isomorphic to $\calM_{0,k}$. It follows that there is a degeneration bringing an arbitrary collection of points together. 

    Now since $\calH^\circ$ has dimension at least $2$, the boundary divisor $D_i$ has dimension at least $1$, and hence one of the vertices of $\Gamma$ has valence at least $4$ (counting marked points).  
    We break up the argument into two cases: 
    \begin{enumerate}[(1)]
    
        \item {\em Bottom vertex has valence $\ge 4$.}  See \Cref{fig:ProduceGenus} for an example of such $D_{i}$.  There is a degeneration for the bottom vertex where two poles corresponding to upward edges (labeled~$e,f$ in the figure) come together. The new graph $\Gamma'$ has a new vertex $v'$ (also of genus zero), which will be connected by edges $e$ and $f$ to the top vertex.  These edges correspond to poles of order at least $-2$ for $v'$. The vertex $v'$ is connected by one new edge, labeled~$h$ in the figure, to the remaining vertex, and thus this edge must correspond to a zero for $v'$, as the sum of the orders of zeroes and poles is $-2$. Thus the vertex $v'$ must be at a level between the original top and bottom levels, which we thus label as $-0.5$. 

       \begin{figure}[h]
        \centering
        \includegraphics[scale=0.7]{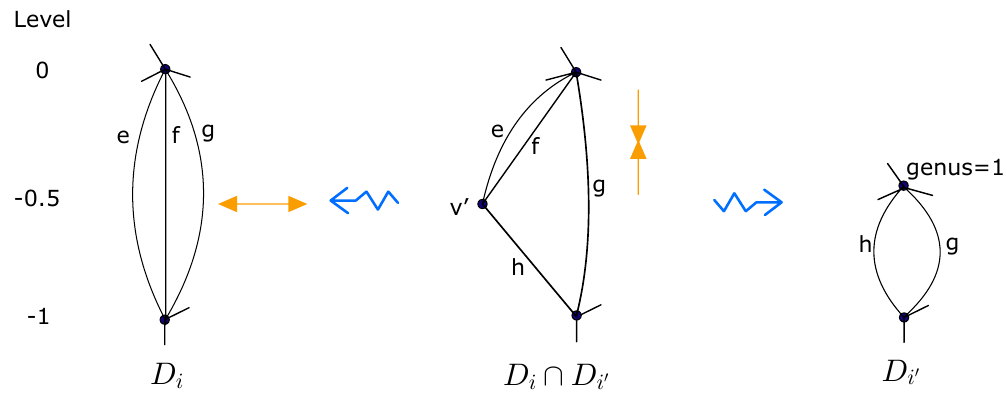}
        \caption{Degeneration and undegeneration that produces a vertex with positive genus.}
        \label{fig:ProduceGenus}
        \end{figure}
    
        Consider then the undegeneration that collapses the level transition between the levels~$0$ and $-0.5$; as it contracts one loop formed by the edges~$e,f$, between two genus zero vertices, in the new graph $\Gamma'$ the new top level vertex will have genus~1.  Then by \Cref{lm:pos-genus} we know that there is a path from the corresponding $D_{i'}$ to $\Dhirr$, and we are done.

        \item {\em Top vertex has valence $\ge 4$.}  There is a degeneration for the top vertex where two points of order $\ge 0$ corresponding to downward edges $e,f$ come together. The new graph $\Gamma'$ has a new vertex $v'$ (also of genus zero), which will be connected by edges $e$ and $f$ to the bottom vertex.  These edges correspond to points of order $\ge 0$ for $v'$. The vertex $v'$ is connected by one new edge $h$ to the remaining vertex, and thus this edge must correspond to a pole for $v'$ which is at least of second order, as the sum of the orders of zeroes and poles is $-2$. Thus~$v'$ must be at a level between the original top and bottom levels.

        From here the argument is the same as in previous case.  
    \end{enumerate}
 \end{proof}

We now complete the proof of the main theorem, by using the same method as in the holomorphic case, and then when \Cref{claim:C} fails, separately dealing as above with the special case of all components of the stable curve being of genus zero. This special case just does not appear in the holomorphic case, as there do not exist genus zero holomorphic strata. 

\begin{proof}[Proof of Main Theorem]
We divide into cases by genus. 
 
    \medskip 
    
    \underline{Case $g>1$.} Recall that $D_j^h$ denote the irreducible components of the horizontal boundary, and $D_i$ denote the irreducible components of the vertical boundary of $\pHo$. By~\Cref{lm:Dhirr_exists}, the irreducible type horizontal boundary divisor $\Dhirr\subset\pHo$ is non-empty, and by~\Cref{thm:Dh_isirr}, $\Dhirr$ is irreducible. By~\Cref{lm:Dh_inter_horiz}, $\Dhirr$ has a non-empty intersection with every irreducible component $D_j^h$ of the horizontal boundary of $\oHo$. 

    By \Cref{lm:horiz-vert}, for any $D_i$, there is a path in $\pHo$ connecting $D_i$ to $\Dhirr$.  Together with the previous paragraph, this implies that $\pHo$ is connected.  Then by \Cref{lm:bdy-cpx-end}, we conclude that $\calH^\circ$ has one end. 

    \medskip
    
    \underline{Case $g=1$.}  Since $\dim \calH^\circ\ge 2$, we have that $n\ge 3$. Then, as in the higher genus case, \Cref{lm:Dh_inter_horiz} and \Cref{lm:horiz-vert} give that  every irreducible boundary divisor of~$\pHo$ is connected to some irreducible component of~ $\Dhirr$.  But now recall that \Cref{thm:Dh_isirr}, due to Lee-Wong, does not apply for $g=1$, so the argument here diverges from that in higher genus.  By \Cref{prop:Dhirrgenu1}, which is proven by Myeongjae Lee in the appendix below, any pair of divisors in $\Dhirr$ can be connected, and we are done.  

    \medskip
    
    \underline{Case $g=0$.} Note that the stratum $\calH_0(m_1,\dots,m_n)$ is isomorphic to $\calM_{0,n}$.  Since $\dim \calH^\circ \ge 2$, we must have $n\ge 5$. It is easy to show that for such $n$, the moduli space $\calM_{0,n}$ is connected and has one end (by e.g. studying the boundary of $\overline{\calM_{0,n}}$).

\end{proof}

\appendix
\section{The horizontal irreducible boundary divisor in genus one}
\begin{center}by \textsc{Myeongjae Lee}\end{center}

In this appendix we show that for~$g=1$ the different irreducible components of $\Dhirr$ are connected to each other in the boundary of the stratum. Recall that for $g>1$ the divisor $\Dhirr\subset\pHo$ is irreducible, see \Cref{thm:Dh_isirr}. By \cite[Thm~1.4 and Thm~1.10]{lewo}, the connected components of $\calH_0(m_1,\dots,m_n\mid -1,-1)$ (i.e.~the locus in the stratum where the residues at the two simple poles are opposite) with $n\geq 3$ are classified by the \textit{index}~$I$, where $1\leq I\leq d\coloneqq \gcd(m_1,\dots,m_n)$. We denote $\Dhirr[I]\subset\overline{\calH}_1(m_1,\dots,m_n)$ the irreducible component of the boundary with index $I$. We remark that swapping the labeling of two simple poles changes the index from $I$ to $d-I$. That is, $\Dhirr[I]=\Dhirr[d-I]$. By plumbing the horizontal node of a multi-scale differential in $\Dhirr[I]$, we obtain a flat surface with a cylinder obtained from the two half-infinite cylinders. The core curve $\alpha$ and the cross curve $\beta$ of this cylinder form a symplectic homology basis $\{\alpha,\beta\}$ such that $\operatorname{Ind}\alpha = 0$ and $\operatorname{Ind} \beta \equiv I\pmod d$. Recall that the index $\operatorname{Ind}\gamma$ of a curve $\gamma$ is the degree of the Gauss map $G_\gamma:S^1\to S^1$, with respect to the flat metric on the translation surface. So $\Dhirr[I]$ smoothens into the connected component of $\calH_1(m_1,\dots,m_n)$ with rotation number $r\coloneqq\gcd(I,m_1,\dots,m_n)=\gcd(I,d)$; that is, the irreducible horizontal boundary of the connected component of $\calH_1(m_1,\dots,m_n)$ with rotation number~$r$ consists of all  $\Dhirr[I]$ such that $r=\gcd(I,d)$. We can prove that each connected component of $\calH_1(m_1,\dots,m_n)$ has an exactly one end by the following

\begin{prop}\label{prop:Dhirrgenu1}
For any $n\ge 3$, let $\calH^\circ$ be the connected component of the stratum $\calH_1(m_1,\dots,m_n)$ consisting of surfaces with rotation number $r\mid d=\gcd(m_1,\dots,m_n)$. Then for any indices~$1\leq I,J\leq d$ such that $r=\gcd(I,d)=\gcd(J,d)$, the irreducible components $\Dhirr[I],\Dhirr[J]$ of~$\pHo$ are connected within $\pHo$.
\end{prop}
\begin{proof}
We will prove that $\Dhirr[I]$ is connected to~$\Dhirr[r]$ within $\pHo$. Then $\Dhirr[I]$ and $\Dhirr[J]$ are also connected within $\pHo$ through $\Dhirr[r]$. 

If $m_i=\pm 1$ for some $i$, then $d=1$, and there exists a unique component $\Dhirr[1]$ of $\Dhirr$, so that the proposition is obvious. Also if $I=d$, then we have $J=d$ and there is nothing to prove. 

So we assume that $|m_i|>1$ for all $i$, and that $I<d$. For each $i$, we claim that there exists an irreducible boundary divisor $D_i\subset\pHo$ that intersects both $\Dhirr[I]$ and $\Dhirr[\gcd(I,m_i)]$. Then $\Dhirr[I]$ and $\Dhirr[\gcd(I,m_i)]$ are connected within $\pHo$ by going along~$D_i$. By repeating this for all $i$, $\Dhirr[I]$ and $\Dhirr[\gcd(I,m_1,\dots,m_n)]=\Dhirr[r]$ are connected within $\pHo$, completing the proof. 

Suppose first that $m_i>1$. Let $D_i$ be the boundary divisor consisting of two-level multi-scale differentials $\overline{X}$ satisfying the following: 
\begin{enumerate}
    \item $\overline{X}$ has two irreducible components at two different levels intersecting at one node. 
    \item The top level component $X_0$ is contained in the stratum $\calH_0(m_1,\dots,m_i-2,\dots, m_n)$ (which is connected).
    \item The bottom level component $X_{-1}$ is contained in the connected component $\mathcal{C}_{\gcd(I,m_i)}$ of $\calH_1(m_i,-m_i)$ with rotation number $\gcd(I,m_i)$.
\end{enumerate} 

Since $\gcd(I,m_i)\leq I<d\leq m_i$, the connected component $\mathcal{C}_{\gcd(I,m_i)}$ exists uniquely by \cite[Thm~1.1]{BoissyMeromCC}, which says that the stratum $\calH_1(m_i,-m_i)$ has a unique connected component $\mathcal{C}_R$ with rotation number $R$ for each $R\mid m_i$ with $R<m_i$. Using the relation between the index and the rotation number, we see that this boundary stratum $D_i$ lies in the boundary of $\calH^\circ$.  

\begin{figure}[!h]
\includegraphics[scale=0.84]{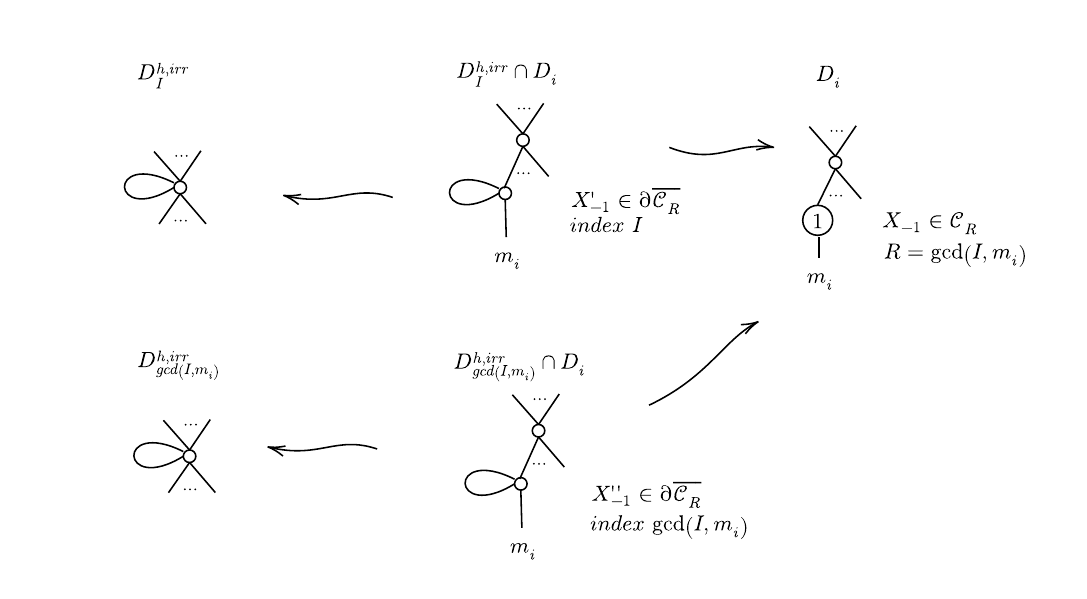}
    \caption{Construction of $D_i$ for the case $m_i>1$.}
    \label{fig:Appendix}
\end{figure}

The horizontal boundary divisors of $\mathcal{C}_{\gcd(I,m_i)}$ arise from the connected components of $\calH_0(m_i,-m_i \mid -1,-1)$ with indices $1\leq t\leq m_i$ such that $\gcd(I,m_i)=\gcd(t,m_i)$. In particular, two flat surfaces $X'_{-1}, X''_{-1}\in \calH_0(m_i,-m_i \mid -1,-1)$ with indices $I$ and $\gcd(I,m_i)$, respectively, are contained in the boundary of the same connected component $\mathcal{C}_{\gcd(I,m_i)}$. These flat surfaces are contained in the intersections $\Dhirr[I]\cap D_i$ and $\Dhirr[\gcd(I,m_i)]\cap D_i$, respectively. See \Cref{fig:Appendix}.

If $m_i<-1$, then similarly let $D_i$ be the boundary divisor consisting of $\overline{Y}$ satisfying:

\begin{enumerate}
    \item $\overline{Y}$ has two irreducible components at two different levels intersecting at one node. 
    \item The bottom level component $Y_{-1}$ is contained in the connected stratum $\calH_0(m_1,\dots,m_i-2,\dots,m_n)$.
    \item The top level component $Y_0$ is contained in the connected component $\mathcal{C}_{\gcd(I,m_i)}$ of $\calH_1(m_i,-m_i)$ with rotation number $\gcd(I,m_i)$.
\end{enumerate} 

By the similar argument as in the previous paragraph with the top level component $Y_0$, we can conclude that $D_i$ intersects both $\Dhirr[I]$ and $\Dhirr[\gcd(I,m_i)]$.
\end{proof}


\begin{thebibliography}{BCG{\etalchar{+}}19}

\bibitem[BCG{\etalchar{+}}19]{BCGGMmsds}
M.~{Bainbridge}, D.~{Chen}, Q.~{Gendron}, S.~{Grushevsky}, and M.~{M{\"o}ller},
  \emph{{The moduli space of multi-scale differentials}}, arXiv e-prints
  (2019), arXiv:1910.13492.

\bibitem[BDG22]{BeDoGr}
F.~Benirschke, B.~Dozier, and S.~Grushevsky, \emph{Equations of linear
  subvarieties of strata of differentials}, Geom. Topol. \textbf{26} (2022),
  no.~6, 2773--2830. \MR{4521253}

\bibitem[Boi12]{BoissyEnds}
C.~Boissy, \emph{Ends of strata of the moduli space of quadratic
  differentials}, Geom. Dedicata \textbf{159} (2012), 71--88. \MR{2944521}

\bibitem[Boi15]{BoissyMeromCC}
\bysame, \emph{Connected components of the strata of the moduli space of
  meromorphic differentials.}, {Comment. Math. Helv.} \textbf{90} (2015),
  no.~2, 255--286.

\bibitem[CGP21]{cgp}
M.~Chan, S.~Galatius, and S.~Payne, \emph{Tropical curves, graph complexes, and
  top weight cohomology of {$\mathcal{M}_g$}}, J. Amer. Math. Soc. \textbf{34}
  (2021), no.~2, 565--594. \MR{4280867}

\bibitem[CMZ22]{euler}
M.~Costantini, M.~M\"oller, and J.~Zachhuber, \emph{The {C}hern classes and
  {E}uler characteristic of the moduli spaces of {A}belian differentials},
  Forum Math. Pi \textbf{10} (2022), Paper No. e16, 55. \MR{4448178}

\bibitem[KZ03]{kozo1}
M.~Kontsevich and A.~Zorich, \emph{Connected components of the moduli spaces of
  {A}belian differentials with prescribed singularities}, Invent. Math.
  \textbf{153} (2003), no.~3, 631--678. \MR{2000471 (2005b:32030)}

\bibitem[Lee24]{Lee}
Myeongjae Lee, \emph{Connected components of strata of residueless meromorphic
  differentials}, Geom. Dedicata \textbf{218} (2024).

\bibitem[LW25]{lewo}
M.~Lee and Y.-M. Wong, \emph{Connected components of generalized strata of
  meromorphic differentials with residue conditions}, preprint
  arXiv:2504.20165.

\bibitem[Tah18]{taharchamber}
Guillaume Tahar, \emph{Chamber structure of modular curves {$X_1(N)$}}, Arnold
  Math. J. \textbf{4} (2018), no.~3-4, 459--481. \MR{3949813}

\end{thebibliography}
\newcommand{\etalchar}[1]{$^{#1}$}
\providecommand{\bysame}{\leavevmode\hbox to3em{\hrulefill}\thinspace}
\providecommand{\MR}{\relax\ifhmode\unskip\space\fi MR }
\providecommand{\MRhref}[2]{%
  \href{http://www.ams.org/mathscinet-getitem?mr=#1}{#2}
}
\providecommand{\href}[2]{#2}

\end{document}